\documentclass[11pt,reqno,oneside]{amsart}
\usepackage[colorlinks,
linkcolor=blue,
anchorcolor=blue,
citecolor=blue,
]{hyperref}
\usepackage{slashed}
\usepackage{amscd}
\usepackage{amsmath}
\usepackage{latexsym}
\usepackage{amsfonts}
\usepackage{amssymb}
\usepackage{amsthm}
\usepackage{graphicx}
\usepackage{verbatim}
\usepackage{mathrsfs}
\usepackage{enumerate}
\usepackage{fullpage}
\usepackage{xcolor}
\usepackage{colonequals}
\usepackage{amsthm}

\newtheorem{theorem}{Theorem}[section]
\newtheorem{corollary}[theorem]{Corollary}
\newtheorem{lemma}[theorem]{Lemma}
\newtheorem{proposition}[theorem]{Proposition}
 \theoremstyle{definition}

	\newtheorem{remark}{Remark}[section]

\numberwithin{equation}{section}

\begin{document}
	
    \title{On Berger's isoperimetric problem}

	\author{Fan Kang}
\address{School of Mathematical Sciences, Capital Normal University, 100048, Beijing, China}
	\email{fankangmath@163.com}

	\begin{abstract}
Berger's isoperimetric problem asks if the flat equilateral torus is $\lambda_1$-maximal. In 1996, Nadirashvili first gave a positive answer. In this paper, we use El Soufi-Ilias-Ros's method in \cite{el1996premiere} and Bryant's result in \cite{bryant2015conformal} to give a new proof.
	\end{abstract}
	
	\maketitle

	\section{Introduction}

 Let $M$ be a Riemannian surface endowed with a Riemannian metric $g$, and denote the corresponding measure by $dv_g$. Let $\Delta_g$ be the associated Laplace-Beltrami operator on $M$ with the first positive eigenvalue $\lambda_1$. For any metric $g$ on $M$,  we know that the restriction of $\lambda_1(g)A(g)$ to the conformal class $[g]=\{\omega g: \omega>0\}$ of $g$ is bounded (see El Soufi-Ilias \cite{el1986immersions}, Li-Yau \cite{yau1982seminar}), where $A(g)$ denotes the Riemannian volume of $g$. Then, denote by 
 \[
 \Lambda_1(M,[g])=\sup\limits_{g'\in [g]}\lambda_1(g')A(g'),\quad \text{and}\quad \Lambda_1(M)=\sup\limits_{[g]}\Lambda_1(M,[g])
 \]
 the first conformal eigenvalue of $[g]$ and the first topological eigenvalue of $M$, respectively. If there exists a smooth metric $g$ on $M$ such that $\lambda_1(g)A(g)=\Lambda_1(M)$,  we call $(M,g)$ a $\lambda_1$-maximal manifold.

Some global maximizers of $\Lambda_1(M)$ have been known. For example, the maximizer on the sphere is given by the round metric by Hersch \cite{hersch1970quatre}, the real projective plane by the constant curvature metric by Li-Yau \cite{li1982new}. The maximizer on the torus $T$ is given by the equilateral torus, which is conjectured by Berger \cite{berger1973premieres} and confirmed by Nadirashvili \cite{nadirashvili1996berger}. Nadirashvili used methods of calculus of variations to prove the existence of a maximizer, showing that the maximal metrics has to correspond to minimal immersions into the sphere. Also, he applied Montiel-Ros's argument \cite{montiel1986minimal} to conclude that the maximal metric is flat. However, Nadirashvili and Montiel-Ros's argument hold only for smooth metrics. In 2019, Cianci-Karpukhin-Medvedev \cite{cianci2019branched} ruled out the maximizers of $\Lambda_1(T)=\frac{8\pi^2}{\sqrt{3}}$ from a larger class of metrics. On the surface of genus two, it is conjectured by Jakobson-Levitin-Nadirashvili-Nigam \cite{jakobson2005large}, and confirmed by Nayatani-Shoda \cite{nayatani2019metrics}, that the maximizer is given by a metric on the Bolza surface with constant curvature one and six conical singularities.
For the Klein bottle, we can refer to \cite{nadirashvili1996berger}, \cite{jakobson2006extremal}, \cite{el2006unique}, \cite{cianci2019branched}. 

In this paper, we give a new proof of the following theorem.
  \begin{theorem}[\textup{\cite{nadirashvili1996berger}}\label{Berger Thm}]
$\Lambda_{1}(T)=\frac{8\pi^2}{\sqrt{3}}$, and the flat equilateral torus is the only $\lambda_{1}$-maximal surface.
\end{theorem}

Recall that any torus is conformally equivalent to a flat torus $(T_{\Gamma}=\mathbb{R}^2/\Gamma,g_{\Gamma})$, where $\Gamma$ is a 2-dimensional lattice and $g_{\Gamma}$ is the metric on $\mathbb{R}^2/\Gamma$ induced from the Euclidean metric on $\mathbb{R}^2$.  Up to isometry and dilations, there is a one-to-one correspondence between the moduli space of flat tori $T(a,b)=\mathbb{R}^2/\Gamma(a,b)$, where $\Gamma(a,b)$ is the lattice $\mathbb{Z}(1,0)\oplus\mathbb{Z}(a,b)$, and the fundamental region 
\begin{equation}\label{Moduli space}
    \mathscr{M}=\Big\{(a,b)\in\mathbb{R}^2:0\leq a \leq \frac{1}{2},\,b\geq\sqrt{1-a^2}\Big\}.
\end{equation}
In addition, it is well known that the eigenvalues of the Laplacian of $(T(a,b),g_{\Gamma})$ are given by
\[
\lambda^{a,b}_{pq}=4\pi^2\Big\{q^2+\left(\frac{p-q a}{b}\right)^2\Big\}
\]
for $(p,q)\in\big\{(p,q)\in\mathbb{Z}\times\mathbb{Z}: q\geq 0 \text{ or } q=0 \text{ and } p\geq 0\big\}$. The induced functions on $T(a,b)$ from
\begin{equation*}
  \begin{aligned}
     & f^{a,b}_{pq}(x,y)=\cos{2\pi\langle \Big(q,\frac{p-q a}{b}}\Big),(x,y)\rangle,\\
     & g^{a,b}_{pq}(x,y)=\sin{2\pi\langle \Big(q,\frac{p-q a}{b}}\Big),(x,y)\rangle
  \end{aligned}  
\end{equation*}
are eigenfunctions associated to $\lambda_{pq}$.

For case $a^2+b^2=1$, the first conformal eigenvalue of these tori is maximized by the flat metric by El Soufi-Ilias-Ros \cite{el1996premiere}, El Soufi-Ilias \cite{el2003extremal}. In \cite{el1996premiere}, El Soufi-Ilias-Ros also gave an upper bound of the first conformal eigenvalue on tori in a fixed conformal class. The proof of their estimates in \cite{el1996premiere} relies on the conformal area of tori, a notion introduced by Li-Yau \cite{yau1982seminar}. In \cite{montiel1986minimal}, Montiel-Ros gave the conformal area of some rectangular tori and conjectured that the conformal area of the tori $T(a,b)$ is precisely $\frac{4\pi^2b}{1+b^2+a^2-a}$, when $(a-\frac{1}{2})^2+b^2<\frac{9}{4}$. The conjecture was confirmed by Bryant \cite{bryant2015conformal}, and he also extended it to the case  $(a-\frac{1}{2})^2+b^2=\frac{9}{4}$. Furthermore, the supremum of the area functional (see the precise definition in Section 2) for conformal immersions $\psi_{ab}$ from $T(a,b)$ into $S^5$ is shown in \cite{bryant2015conformal} as
\begin{equation}\label{sup of A}
\sup_{\gamma} A(\gamma \circ \psi_{ab}) = \frac{8\pi^2 b \sqrt{b^2 + a^2 - a + 1}}{3\sqrt{3}(b^2 + a^2 - a)},
\end{equation}
when $(a - \frac{1}{2})^2 + b^2 > \frac{9}{4}$. Here, $\gamma$ ranges over all conformal transformations of $S^5$ and $\psi_{ab}$ is defined by 
\begin{equation*}
\begin{aligned}
\psi_{ab}=&\frac{1}{\sqrt{1+b^2+a^2-a}}\Big(\sqrt{b^2+a^2-a}f^{a,b}_{10}, \sqrt{b^2+a^2-a}g^{a,b}_{10}, \\
&\sqrt{1-a}f^{a,b}_{01},\sqrt{1-a}g^{a,b}_{01},\sqrt{a}f^{a,b}_{11}, \sqrt{a}g^{a,b}_{11}\Big).
\end{aligned}
\end{equation*} In this paper, we focus only on the case of tori $T(0,b)$, and use the result \eqref{sup of A} to provide a new proof of Theorem \ref{Berger Thm}. As a by-product, we obtain an upper bound for the first conformal eigenvalue on tori. That is,
\begin{corollary}\label{eigenvalue-estimate}
    For all $(a,b)\in\mathscr{M}$, we have
    \begin{equation}\label{eigenvalue estimation}
    \Lambda_1(T_{\Gamma},[g_{\Gamma}])\leq \frac{8\pi^2}{\sqrt{6}b}\frac{\sqrt{2+a^2+b^2+L}}{a^2+b^2+L}
\Big(a^2+b^2+\frac{L}{3} \Big)
\end{equation}
with $L:=\sqrt{(a^2+b^2)(8+a^2+b^2)}$.
Moreover, when $a^2+b^2>1,$ the inequality is strict.
\end{corollary}

\begin{remark}
This corollary gives a better upper bound of $\Lambda_1(T_{\Gamma},[g_{\Gamma}])$ compared to the one of El Soufi-Ilias-Ros in \cite{el1996premiere}. To be more specific, in \cite{el1996premiere}, the authors showed that for all $(a,b)\in\mathscr{M}$ and $g$ conformally equivalent to $g_{\Gamma}$,  
\begin{equation}\label{el-estimate}  
\lambda_1(g)A(g)\leq\frac{3\pi^2}{2b}\big(a^2+b^2+\frac{5}{3}\big). 
\end{equation}
Our result improves this bound, as the right-hand side of \eqref{eigenvalue estimation} is smaller than that of \eqref{el-estimate}.
\end{remark}
    \section{Proof of the main theorem}

    Let $S^n$ be the n-dimensional unit sphere, $D^{n+1}$ the open unit ball bounded by $S^n$ in $\mathbb{R}^{n+1}$, and $G(n)$ the conformal group of $S^n$. For each point $\gamma\in D^{n+1}$, we consider the map, also denoted by $\gamma$, $\gamma\colon S^n \to S^n$ given by
    \begin{equation}\label{map gamma}
    \gamma(p)=\frac{p+\big(\beta\langle p,\gamma \rangle+\alpha\big)\gamma}{\alpha\big(\langle p,\gamma \rangle+1\big)}
    \end{equation}
    for all $p\in S^n$, where $\alpha=\frac{1}{\sqrt{1-|\gamma|^2}}$, $\beta=(\alpha-1)\frac{1}{|\gamma|^2}$ and $\langle \, ,\, \rangle$ denotes the usual inner product on $\mathbb{R}^{n+1}$. Then $\gamma$ is a conformal transformation of $S^{n}$, and each transformation of $G(n)$ can be expressed by $\xi\circ \gamma$ for some orthogonal transformation $\xi$. 

Let $M$ be a compact surface. For each branched conformal immersion $\psi\colon M\to S^n$, we can consider the area function $A\colon G(n)\to \mathbb{R}$, which maps a conformal transformation $\gamma$ of $S^n$ on the area induced from the immersion $\gamma\circ \psi$. We will denote it by $A(\gamma\circ \psi)$. Since the area function restricts to conformal transformations of the type \eqref{map gamma}, the area function can be defined on the unit ball, $A\colon D^{n+1}\to\mathbb{R}$. Then, for each $\gamma\in D^{n+1}$, we obtain
        \begin{equation}\label{area function}
        A(\gamma\circ \psi)=\frac{1}{2}\int_{M}\frac{1-|\gamma|^2}{\big(1-\langle\psi,\gamma\rangle\big)^2}|\nabla\psi|^2\,dM.
        \end{equation}
Note that for the identity map $\text{id}:S^n\to S^n$, which corresponds to the origin $0$ in $D^{n+1}$, we simply denote $A(\text{id} \circ\psi)=A(\psi)$. As described above, for the immersion $\psi_{b}:=\psi_{0b}$, when $\gamma$ ranges over all conformal transformations, the supremum of the area functional $A(\gamma\circ\psi_{b})$ has already been provided by Bryant \cite{bryant2015conformal}. For completeness, we restate them here.
   
\begin{lemma}[\textup{\cite{bryant2015conformal}}]\label{main lemma}
Let $M$ be a torus conformally equivalent to $T_{b}=T(0,b)$, and let $\psi_{b}\colon T_b\to S^3$ be the following immersion:
\[
\psi_{b}(x,y)=\frac{1}{\sqrt{1+b^2}}\Big(b\cos\frac{2\pi y}{b},b\sin\frac{2\pi y}{b}, \cos{2\pi x},\sin{2\pi x}\Big).
\]
Then, for $1\leq b \leq\sqrt{2}$, we have
\begin{equation}\label{case 1}
\sup\limits_{\gamma\in G(3)}A(\gamma\circ\psi_{b})=\frac{4\pi^2 b}{1+b^2};
\end{equation}
for $b>\sqrt{2}$, 
\begin{equation}\label{case 2}
  \sup\limits_{\gamma\in G(3)}A(\gamma\circ\psi_{b})= \frac{8\pi^2\sqrt{b^2+1}}{3\sqrt{3}b}.
\end{equation}
\end{lemma}
\begin{proof}
For completeness and to provide more computation details, we give the proof below. 
We basically followed Bryant's idea in \cite{bryant2015conformal}. 

To begin with, we substitute $\psi_{b}$ into equation \eqref{area function}, and denote $r_1=\frac{b^2}{1+b^2}$, $r_2=\frac{1}{1+b^2}$, $t=\frac{2\pi y}{b}$, and $s=2\pi x$. For any $\gamma=(a_1,\dots,a_4)\in\mathbb{R}^4$ with $|\gamma|^2<1$, we have
\begin{equation}\label{area-ftn1}
A(\gamma \circ \psi_b) = \frac{b}{1+b^2} \int_{0}^{2\pi} \int_{0}^{2\pi} \frac{1 - |\gamma|^2}{\big(1 - \langle \psi_b(s,t),\gamma \rangle \big)^2} \,ds\,dt.
\end{equation}
Since $\psi_b$ is equivariant with respect to torus of rotations in $SO(4)$, one can apply a rotation from this torus to reduce to the cases in which $a_1,a_3\geq 0$ and $a_2=a_4=0$.
Setting $\lambda=a_1\sqrt{r_1}$, $\mu=a_3\sqrt{r_2}$. Simplifying equation \eqref{area-ftn1} gives
\begin{equation}\label{area-ftn2}
    \begin{aligned}
 A(\gamma\circ\psi_b)=\frac{b}{1+b^2}\Big(1-\frac{\lambda^2}{r_1}-\frac{\mu^2}{r_2}\Big)\int_{0}^{2\pi}\int_{0}^{2\pi}\frac{ds\,dt}{\big(1-\lambda\cos{t}-\mu\cos{s}\big)^2}.
\end{aligned}
\end{equation}
Now, we have simplified the area functional  $A(\gamma \circ \psi_b)$ into a double integral expression with respect to $s$ and $t$, where $\lambda$ and $\mu$ are parameters. To find the supremum of the area functional $A(\gamma \circ \psi_b)$ within the region $|\gamma|^2 < 1$, we transform the problem into finding the supremum of the right-hand side of \eqref{area-ftn2} over the region 
\[
\Omega=\left\{(\lambda, \mu)\in\mathbb{R}^2: \frac{\lambda^2}{r_1}+\frac{\mu^2}{r_2}<1, \, \lambda\geq0, \, \mu\geq0\right\}.
\]
That is,
\[
\sup_{\gamma \in G(3)} A(\gamma \circ \psi_b) = \sup_{(\lambda, \mu) \in \Omega} \frac{b}{1 + b^2} \left(1 - \frac{\lambda^2}{r_1} - \frac{\mu^2}{r_2}\right) \int_{0}^{2\pi} \int_{0}^{2\pi} \frac{ds \, dt}{(1 - \lambda \cos t - \mu \cos s)^2}.
\]
Next, we compute this double integral
\[
\int_{0}^{2\pi}\int_{0}^{2\pi}\frac{ds\,dt}{\big(1-\lambda\cos{t}-\mu\cos{s}\big)^2}.
\]
We use the following two definite integral formulas. 
The first formula is  
\[
\int_{0}^{2\pi}\frac{ds}{(a-b\cos{s})^2} = \frac{2\pi a}{(a^2 - b^2)^{3/2}},
\]
where $a > |b|$. The second formula is  
\[
\int_{0}^{2\pi} \frac{(1 - a\cos{t})\,dt}{\big( (1 - a\cos{t})^2 - b^2 \big)^{3/2}} = \frac{4 E\left( \sqrt{4ab/\big(1 - (a-b)^2\big)} \right)}{\Big(1 - (a-b)^2\Big)^{1/2}\Big(1 - (a+b)^2\Big)},
\]
where $a, b \geq 0$, $a + b < 1$, and $E(k) = \int_{0}^{\pi/2} \sqrt{1 - k^2 \sin^2{\theta}}\, d\theta$, with $k \in [-1,1]$ being the complete elliptic integral of the second kind. The details of the second formula can be found in the appendix of Bryant \cite{bryant2015conformal}. Substituting these formulas into equality \eqref{area-ftn2}, we obtain
\begin{equation}\label{area-ftn3}
A(\gamma\circ\psi_b)=\frac{8\pi b}{1+b^2}\Big(1-\frac{\lambda^2}{r_1}-\frac{\mu^2}{r_2}\Big)\frac{E\left(\sqrt{4\lambda\mu/\big(1-(\lambda-\mu)^2\big)}\right)}{\Big(1-(\lambda-\mu)^2\Big)^{1/2}\Big(1-(\lambda+\mu)^2\Big)}.
\end{equation}
Since the complete elliptic integral of the second kind $E(k)$ can be expressed as a power series:
\[
E(k) = \frac{\pi}{2} \sum_{n=0}^{\infty} \left( \frac{(2n)!}{2^{2n}(n!)^2} \right)^2 \frac{k^{2n}}{1 - 2n},
\]
we have
\[
E(k) \leq \frac{\pi}{2} \left(1 - \frac{1}{4} k^2 \right).
\]
Substituting this inequality into equation \eqref{area-ftn3}, we get
\begin{equation*}
    \begin{aligned}
A(\gamma\circ\psi_b)\leq \frac{4\pi^2b}{1+b^2}\frac{\Big(1-\frac{\lambda^2}{r_1}-\frac{\mu^2}{r_2}\Big)\Big(1-(\lambda+\mu)^2+3\lambda\mu\Big)}{\Big(1-(\lambda-\mu)^2\Big)^{3/2}\Big(1-(\lambda+\mu)^2\Big)}.
\end{aligned}
\end{equation*}
Let \[
I(\lambda,\mu)=\frac{\Big(1-\frac{\lambda^2}{r_1}-\frac{\mu^2}{r_2}\Big)\Big(1-(\lambda+\mu)^2+3\lambda\mu\Big)}{\Big(1-(\lambda-\mu)^2\Big)^{3/2}\Big(1-(\lambda+\mu)^2\Big)},\quad(\lambda,\mu)\in\Omega.
\]
If we prove that 
\begin{equation}\label{main-step}
\sup\limits_{(\lambda,\mu)\in\Omega}\frac{4\pi^2b}{1+b^2}I(\lambda,\mu)=
\begin{cases} 
\frac{4\pi^2b}{1+b^2}, & \text{if } 1\leq b \leq\sqrt{2}, \\ \\[0pt] 
\frac{8\pi^2\sqrt{b^2+1}}{3\sqrt{3}b}, & \text{if } b>\sqrt{2},
\end{cases}
\end{equation}
based on the facts that $A(\psi_b)=\frac{4\pi^2b}{1+b^2}$ and  
$A(\gamma_1\circ\psi_b)=\frac{8\pi^2\sqrt{b^2+1}}{3\sqrt{3}b}$ for $\gamma_1=(\sqrt{3r_1-2},0,0,0)$,  
then the conclusion follows.
To prove \eqref{main-step}, we consider two cases separately. 

\textbf{Case 1: for \(1\leq b \leq \sqrt{2}\).} Noting that \(I(0,0)=1\), it suffices to prove that \(I(\lambda,\mu) \leq 1\) for all \((\lambda,\mu) \in \Omega\). Moreover, since \((1-c)^{3/2}\geq 1-\frac{3}{2}c\) when \(0\leq c\leq 1\), it remains only to show that
\[
Q(\lambda,\mu):=\Big(1-\frac{3}{2}(\lambda-\mu)^2\Big)\Big(1-(\lambda+\mu)^2\Big)-\Big(1-\frac{\lambda^2}{r_1}-\frac{\mu^2}{r_2}\Big)\Big(1-(\lambda+\mu)^2+3\lambda\mu\Big)\geq 0.
\]
We simplify $Q(\lambda,\mu)$ and obtain 
\begin{equation*}
    \begin{aligned}
        Q(\lambda,\mu)&=\Big(1-\frac{3}{2}(\lambda^2+\mu^2)+3\lambda\mu\Big)\Big(1-(\lambda+\mu)^2\Big)-\Big(1-\frac{\lambda^2}{r_1}-\frac{\mu^2}{r_2}\Big)\Big(1-(\lambda+\mu)^2+3\lambda\mu\Big)\\
        &=\Big(\frac{\lambda^2}{r_1}+\frac{\mu^2}{r_2}-\frac{3}{2}(\lambda^2+\mu^2)\Big)\Big(1-(\lambda+\mu)^2\Big)+3\lambda\mu\Big(\frac{\lambda^2}{r_1}+\frac{\mu^2}{r_2}-(\lambda+\mu)^2\Big).
    \end{aligned}
\end{equation*}
Since $1\leq b\leq \sqrt{2}$, it follows that $\frac{1}{2}\leq r_1\leq\frac{2}{3}$. Then, we obtain  
\[
\frac{\lambda^2}{r_1}+\frac{\mu^2}{r_2}-\frac{3}{2}(\lambda^2+\mu^2)\geq 0, \quad \text{for all } (\lambda,\mu)\in\mathbb{R}^2.
\]
And, since the matrix 
\[
 \begin{bmatrix} 
\frac{1}{r_1}-1 & -1 \\ 
-1 & \frac{1}{r_2}-1 
\end{bmatrix}
\]
is non-negative definite, the corresponding quadratic form $\frac{\lambda^2}{r_1}+\frac{\mu^2}{r_2}-(\lambda+\mu)^2$ is non-negative on $\mathbb{R}^2$.
For all $(\lambda,\mu)\in\Omega$, we have $\lambda \mu\geq0$, and $(\lambda+\mu)^2\leq\frac{\lambda^2}{r_1}+\frac{\mu^2}{r_2}< 1$. Then, we have $Q(\lambda,\mu)\geq0$ on $\Omega$. 

\textbf{Case 2: for $b>\sqrt{2}$.} Noting that $\quad  \frac{4\pi^2b}{1+b^2}=\frac{8\pi^2\sqrt{1+b^2}}{3\sqrt{3}b}\frac{3\sqrt{3}r_1\sqrt{1-r_1}}{2}$, it suffices to prove that 
\[
\sup\limits_{(\lambda,\mu)\in\Omega}I(\lambda,\mu)= \frac{2}{3\sqrt{3}r_1\sqrt{1-r_1}}.
\]
Now, we observe that $I(\lambda,\mu)$ can be continuously extended to $\overline{\Omega}$ except for the point $(r_1, r_2)$, where $\overline{\Omega}$ is the closure of $\Omega$. For the discontinuity point $(r_1, r_2)$, it is not hard to show that, for every $\varepsilon>0$, there is a $\delta>0$ such that for $(\lambda,\mu)\in \Omega$ satisfying $(\lambda-r_1)^2+(\mu-r_2)^2<\delta$, the inequalities
\[
0<I(\lambda,\mu)<\frac{3}{8\sqrt{r_1(1-r_1)}}+\varepsilon
\]
hold. Since 
\[
\frac{3}{8\sqrt{r_1(1-r_1)}}<\frac{2}{3\sqrt{3}r_1\sqrt{1-r_1}}
\]
when $2/3<r_1<1$, it follows that the supremum of $I(\lambda,\mu)$ in $\Omega$ must occur in some compact set 
\[
R_{\delta}=\left\{(\lambda,\mu):\,\frac{\lambda^2}{r_1}+\frac{\mu^2}{r_2}\leq 1,\,(\lambda-r_1)^2+(\mu-r_2)^2\geq\delta\text{ and } \lambda\geq0,\,\mu\geq 0\right\},
\]
for some $\delta>0$. 

To find the maximum of $I(\lambda,\mu)$ on $R_{\delta}$, we first determine the maximum value of $I(\lambda,\mu)$ on the boundary of $R_{\delta}$. Let $\Omega_1=\{(\lambda,\mu)\in R_{\delta}:\frac{\lambda^2}{r_1}+\frac{\mu^2}{r_2}=1\}$,
 and it is straightforward that $I(\lambda,\mu)=0$ for all $(\lambda,\mu)\in\Omega_{1}$. Let $\Omega_2=\{(\lambda,\mu)\in R_{\delta}:(\lambda-r_1)^2+(\mu-r_2)^2=\delta\}$. With the choice of $R_{\delta}$ as above, we can see that $I(\lambda,\mu)< \frac{2}{3\sqrt{3}r_1\sqrt{1-r_1}}$, for all $(\lambda,\mu)\in\Omega_2$.
 Let $\Omega_{3}=\{(\lambda,\mu)\in R_{\delta}:\mu=0\}$. For all $(\lambda,\mu)\in\Omega_3$, $I(\lambda,\mu)=\frac{1-\lambda^2/r_1}{(1-\lambda^2)^{3/2}}$. By direct computation, we find that 
 \[
\max_{\Omega_3}I(\lambda,\mu)=I(\sqrt{3r_1-2},0)=\frac{2}{3\sqrt{3}r_1\sqrt{1-r_1}}.
 \]
 Also, let $\Omega_{4}=\{(\lambda,\mu)\in R_{\delta}:\lambda=0\}$. Similar to the previous case, we have $\max\limits_{\Omega_4}I(\lambda,\mu)= I(0,0)=1$ on $\Omega_4$. Thus, due to $\partial R_{\delta}=\cup_{i=1}^{4}\Omega_i$, we obtain
\[
\max\limits_{\partial R_{\delta}}I(\lambda,\mu)=\frac{2}{3\sqrt{3}r_1\sqrt{1-r_1}}.
\]
Next, we claim that $I(\lambda,\mu)$ has no critical point in the interior of $R_{\delta}$. 
To compute the critical points of the function $I$ in the interior of $R_{\delta}$, we first compute its partial derivatives with respect to $\lambda,\,\mu$. That is,
\begin{equation*}
\begin{cases} 
\frac{\partial I(\lambda,\mu)}{\partial\lambda}=\frac{U(r_1,\lambda,\mu)}{r_1r_2\left(1-(\lambda-\mu)^2\right)^{5/2}\left(1-(\lambda+\mu)^2\right)^2}, \\  
\frac{\partial I(\lambda,\mu)}{\partial\mu}=\frac{V(r_1,\lambda,\mu)}{r_1r_2\left(1-(\lambda-\mu)^2\right)^{5/2}\left(1-(\lambda+\mu)^2\right)^2},
\end{cases}
\end{equation*}
where $U(r_1,\lambda,\mu)$ and $V(r_1,\lambda,\mu)$ are polynomials in $r_1,\,\lambda$ and $\mu$. 
This implies that the critical points satisfy $U(r_1,\lambda,\mu)=0$ and $V(r_1,\lambda,\mu)=0$.

To solve equations $U=0$ and $V=0$, one can use MAPLE to compute the Gr\"{o}bner basis of the ideal generated by $U(r_1,\lambda,\mu)$ and $V(r_1,\lambda,\mu)$ with respect to the pure lexicographical order $r_1>\lambda>\mu$. It is found that the first element of this Gr\"{o}bner basis is a polynomial in $\lambda$ and $\mu$ that factors as 
\[
G(\lambda,\mu)=\lambda\mu\left(1-(\lambda-\mu)^2\right)^2\left(1-(\lambda+\mu)^2\right)^3(1-\lambda^2+\lambda\mu-\mu^2)\left(1-(\lambda-\mu)^2(2-\lambda^2-\lambda\mu-\mu^2)\right).
\]
Since $G(\lambda,\mu)>0$ in the interior of $\Omega$, and since $\operatorname{Int}(R_\delta)\subset \Omega$, it follows that there are no critical points of $I$ in the interior of $R_\delta$. The lemma is proved.
\end{proof}

Before proving Theorem \ref{Berger Thm}, let us recall the definition of the energy functional. For any map $F:\big(T(a,b),g\big)\to S^n$, we define the energy functional of $F$ as 
\[
E(F)=\frac{1}{2}\int_{T(a,b)} |\nabla F|^2\,dv_g.
\]
The following proposition from \cite{el1996premiere} describes maps whose energy ratio is invariant under the action of the conformal group of the sphere.

\begin{proposition}[\textup{\cite{el1996premiere}}]
Let $(A_i)_{1\leq i\leq n}$ be real numbers such that $\sum\limits_{i=1}^{n}A_i^2=1$, and let $(p_i,q_i)_{1\leq i\leq n}$ be pairs of integers. For every $(a,b)\in \mathscr{M}$, consider the map $F_{a,b}:T(a,b)\to S^{2n-1}$ given by 
 \[
 F_{a,b}=\Big(A_1f^{a,b}_{p_1q_1},A_1 g^{a,b}_{p_1q_1},\dots, A_nf^{a,b}_{p_nq_n},A_n g^{a,b}_{p_nq_n}\Big).
 \]
 Then, for every $\gamma\in G(2n-1)$, $(a,b)\in \mathscr{M}$ and $(a',b')\in\mathscr{M}$, we have
 \begin{equation}\label{identity}
 \frac{E(\gamma\circ F_{a,b})}{E(\gamma\circ F_{a',b'})}=\frac{E(F_{a,b})}{E(F_{a',b'})}.
\end{equation}
\end{proposition}

For convenience, we denote $\Lambda_1(a,b)$ as the first conformal eigenvalue of the flat torus $\Lambda_1(T(a,b),[g_{\Gamma}])$. Moreover, note that for the conformal map $f$, we have $E(f) = A(f)$.
\begin{proof}[\textbf{Proof of Theorem 1.1}]
Let $\mathscr{M}_1=\Big\{(a,b)\in\mathscr{M}: b\leq \sqrt{2}\Big\}$, and $\mathscr{M}_2=\mathscr{M}\setminus\mathscr{M}_1$. For the torus $T(a,b),$ $(a,b)\in \mathscr M_1$, let $\Phi_{\sqrt{2}}\colon T_{\Gamma}\to S^3$ be the map induced from
    \[\Phi_{\sqrt{2}}(x,y)=\frac{1}{\sqrt{3}}\Big(\sqrt{2}\cos\frac{2\pi y}{b},\sqrt{2}\sin\frac{2\pi y}{b},\cos{2\pi (x-\frac{a}{b}y)},\sin{2\pi (x-\frac{a}{b}y)}\Big).\]
For all $g\in [g_{\Gamma}]$, it follows from Hersch's lemma \cite{hersch1970quatre} that there exists $\gamma\in G(3)$ such that for any $j\leq4$, $\int_{T_\Gamma}(\gamma\circ\Phi_{\sqrt{2}})_j\,dv_g=0$. Using the min-max principle and summing up, we have
\begin{equation*}
\begin{aligned}
\lambda_1(g) A(g) \leq \int_{T_{\Gamma}} |d(\gamma \circ \Phi_{\sqrt{2}})|^2 \, dv_g = 2 E(\gamma \circ \Phi_{\sqrt{2}}).
\end{aligned}
\end{equation*}
By the equality \eqref{identity}, we obtain
\[
E(\gamma \circ \Phi_{\sqrt{2}}) = A(\gamma \circ \psi_{\sqrt{2}}) \cdot \frac{E(\Phi_{\sqrt{2}})}{A(\psi_{\sqrt{2}})}.
\]
Therefore,
\begin{equation*}
\begin{aligned}
\lambda_1(g) A(g) &\leq 2 A(\gamma \circ \psi_{\sqrt{2}}) \cdot \frac{E(\Phi_{\sqrt{2}})}{A(\psi_{\sqrt{2}})}.
\end{aligned}
\end{equation*}
Moreover, by Lemma \ref{main lemma}, we have
\[
\sup\limits_{\gamma\in G(3)}A(\gamma \circ \psi_{\sqrt{2}})=\frac{4\sqrt{2}\pi^2}{3}.
\]
Thus, 
\begin{equation*}
\begin{aligned}
\lambda_1(g) A(g) &\leq 2\sup_{\gamma \in G(3)}A(\gamma \circ \psi_{\sqrt{2}}) \cdot \frac{E(\Phi_{\sqrt{2}})}{A(\psi_{\sqrt{2}})}\\
&=\frac{8\pi^2 \sqrt{2}}{3} \cdot \frac{E(\Phi_{\sqrt{2}})}{A(\psi_{\sqrt{2}})} \\
&= \frac{4\pi^2}{3b} (a^2 + b^2 + 2).
\end{aligned}
\end{equation*}
For the torus $T(a,b),$ $(a,b)\in \mathscr M_2$, let $\Phi_{b}\colon T_{\Gamma}\to S^3$ be the map induced from
    \[\Phi_{b}(x,y)=\frac{1}{\sqrt{1+b^2}}\Big(b\cos\frac{2\pi y}{b},b\sin\frac{2\pi y}{b},\cos{2\pi (x-\frac{a}{b}y)},\sin{2\pi (x-\frac{a}{b}y)}\Big).\]
Similarly to the case above, using $\gamma\circ\Phi_{b}$ as test functions for $\lambda_1$, we obtain
\begin{equation*}
     \begin{aligned}
\lambda_1(g)A(g)&\leq \int_{T_{\Gamma}}|d(\gamma\circ\Phi_{b})|^2\,dv_g=2E(\gamma\circ\Phi_{b})\\
&=2A(\gamma\circ\psi_{b})\cdot \frac{E(\Phi_{b})}{A(\psi_{b})}\\
&\leq 2\sup\limits_{\gamma\in G(3)}A(\gamma\circ\psi_{b})\cdot \frac{E(\Phi_{b})}{A(\psi_{b})}\\
&=\frac{8\pi^2\sqrt{b^2+1}}{3\sqrt{3}}\frac{(2b^2+a^2)}{b^3}.
    \end{aligned}
     \end{equation*} 
Note that $0\leq a\leq\frac{1}{2}.$ Then for any $(a,b)\in \mathscr{M}_1,$ we have
\begin{equation*}
     \begin{aligned}
     \Lambda_1(a,b)\leq\frac{4\pi^2}{3b}(\frac{9}{4}+b^2).
\end{aligned}
     \end{equation*}
For any $(a,b)\in \mathscr{M}_2,$ we have
\begin{equation*}
     \begin{aligned}
     \Lambda_1(a,b)\leq\frac{2\pi^2\sqrt{b^2+1}}{3\sqrt{3}}\frac{8b^2+1}{b^3}.
\end{aligned}
     \end{equation*}
Therefore, we obtain
 \begin{equation*}
     \begin{aligned}
\Lambda_1(T)&=\sup\limits_{(a,b)\in\mathscr{M}}\Lambda_1(a,b)\\
&\leq\sup\limits_{(a,b)\in\mathscr{M}_1}\frac{4\pi^2}{3b}\Big(\frac{9}{4}+b^2\Big)\\
     &=\frac{8\pi^2}{\sqrt{3}}.
     \end{aligned}
     \end{equation*}
The equality holds if and only if $(a,b)=(\frac{1}{2}, \frac{\sqrt{3}}{2})$.
Moreover, by El Soufi-Ilias-Ros \cite{el1996premiere}, El Soufi-Ilias \cite{el2003extremal}, the metric attaining $\Lambda_1(\frac{1}{2}, \frac{\sqrt{3}}{2})$ must be flat. The proof is done.
    \end{proof}

\begin{proof}[\textbf{Proof of Corollary 1.2}]
 Let $\Phi_{b_{0}}\colon T_{\Gamma}\to S^3$ be the map induced from
    \[\Phi_{b_{0}}(x,y)=\frac{1}{\sqrt{1+b_0^2}}\Big(b_0\cos\frac{2\pi y}{b},b_0\sin\frac{2\pi y}{b},\cos{2\pi (x-\frac{a}{b}y)},\sin{2\pi (x-\frac{a}{b}y)}\Big),\]
    where $b_0\in[1,+\infty)$. 
    
For all $g\in [g_{\Gamma}]$, similarly, using $\gamma\circ\Phi_{b_{0}}$ as test functions for $\lambda_1$, we obtain
         \begin{equation}\label{mian inequality}
     \begin{aligned}
\lambda_1(g)V(g)&\leq \int_{T_{\Gamma}}|d(\gamma\circ\Phi_{b_{0}})|^2\,dv_g=2E(\gamma\circ\Phi_{b_{0}})\\
&=2A(\gamma\circ\psi_{b_0})\cdot \frac{E(\Phi_{b_{0}})}{A(\psi_{b_0})}.
    \end{aligned}
     \end{equation} 
When $1\leq b_0\leq\sqrt{2}$, by \eqref{case 1}, we have
      \begin{equation*}
     \begin{aligned}
     \lambda_1(g)V(g)&\leq2\sup\limits_{\gamma\in G(3)}\frac{A(\gamma\circ\psi_{b_0})}{A(\psi_{b_0})}\cdot E(\Phi_{b_{0}})\\
&=2E(\Phi_{b_{0}})   \\
&=\frac{4\pi^2}{b}\frac{({b_0}^2+a^2+b^2)}{1+b_0^2}.
\end{aligned}
     \end{equation*}
     Then,
     \begin{equation*}\label{b-case 1}
     \begin{aligned}
     \lambda_1(g)V(g)&\leq\inf\limits_{1\leq b_0\leq\sqrt{2}}\frac{4\pi^2}{b}\frac{({b_0}^2+a^2+b^2)}{1+b_0^2}\\
&=\frac{4\pi^2}{3b}(a^2+b^2+2).
\end{aligned}
     \end{equation*}
 When $b_0>\sqrt{2}$, by \eqref{case 2}, we have
 \begin{equation}
     \begin{aligned}\label{b-case 2}
    \lambda_1(g)V(g)&\leq2\sup\limits_{\gamma\in G(3)}A(\gamma\circ\psi_{b_0})\cdot \frac{E(\Phi_{b_{0}})}{A(\psi_{b_0})}\\
    &= 2\sup\limits_{\gamma\in G(3)}A(\gamma\circ\psi_{b_0})\cdot\frac{b_0^2+a^2+b^2}{2b_0b}\\
    &= \frac{8\pi^2\sqrt{b_0^2+1}}{3\sqrt{3}b_0^2}\frac{(b_0^2+a^2+b^2)}{b}.
  \end{aligned}
     \end{equation}
Let $F(b_0)=\frac{8\pi^2\sqrt{b_0^2+1}}{3\sqrt{3}b_0^2}\frac{(b_0^2+a^2+b^2)}{b},$ $b_0>\sqrt{2}.$ By the direct computation, we have 
\[
\inf\limits_{b_0>\sqrt{2}}F(b_0)=F(b_{0}^{'}),
\]
where $b_{0}^{'}=\sqrt{a^2+b^2+\sqrt{(a^2+b^2)(8+a^2+b^2)}}/\sqrt{2}.$ Therefore, $\Lambda_1(a,b)\leq F(b_{0}^{'}),$ where $F(b_{0}^{'})$ is the right hand side of \eqref{eigenvalue estimation}.
Then we obtain 
\[
\Lambda_1(a,b)\leq\min\Big\{\frac{4\pi^2}{3b}(a^2+b^2+2),F(b_{0}^{'})\Big\}=F(b_{0}^{'}).
\]

When $a^2+b^2>1,$ the inequality $\eqref{eigenvalue estimation}$ is strict. Otherwise, there exists $g^{'}=\omega g_{\Gamma}$, $\omega>0$, such that $\lambda_1(g^{'})V(g^{'})$ is equal to the right hand side of $\eqref{eigenvalue estimation}$. This implies that equality holds in the inequalities of \eqref{mian inequality} and \eqref{b-case 2}. Then, for $\gamma=(\sqrt{3r_1-2},0,0,0)$ and $\Phi_{b_{0}^{'}}$, we have 
\[
\Delta_{g_{\Gamma}}(\gamma\circ\Phi_{b_{0}^{'}})=\lambda_1(g^{'})\omega(\gamma\circ\Phi_{b_{0}^{'}}).
\]
It is impossible. In fact,
\begin{equation*}
     \begin{aligned}
&(\gamma\circ\Phi_{b_{0}^{'}})_4\Delta_{g_{\Gamma}}{(\gamma\circ\Phi_{b_{0}^{'}})_2}-(\gamma\circ\Phi_{b_{0}^{'}})_2\Delta_{g_{\Gamma}}(\gamma\circ\Phi_{b_{0}^{'}})_4\\
&=\frac{1}{b^2\lambda^2\left(\sqrt{3r_1-2}b_{0}^{'} \cos\frac{2\pi y}{b}+\sqrt{(b_{0}^{'})^{2}+1}\right)^{3}}\Big((a^2+b^2-1)\sqrt{1+(b_{0}^{'})^2}\cos{2\pi(x-\frac{ay}{b})}\\
&+\sqrt{3r_1-2}(a^2+b^2+1)b_{0}^{'}\cos{\frac{2\pi y}{b}}\cos{2\pi(x-\frac{ay}{b})}-2a\sqrt{3r_1-2}b_{0}^{'}\sin{\frac{2\pi y}{b}}\sin{2\pi(x-\frac{ay}{b})}\Big)\\
&\neq 0.
\end{aligned}
\end{equation*}
\end{proof}

\begin{remark}
    By direct computation, for all $a^2+b^2>1,$ $\gamma\in D^4,$ and $b_0\in [1,+\infty),$ we have $(\gamma\circ\Phi_{b_{0}})_4\Delta_{g_{\Gamma}}{(\gamma\circ\Phi_{b_{0}})_2}\neq(\gamma\circ\Phi_{b_{0}})_2\Delta_{g_{\Gamma}}(\gamma\circ\Phi_{b_{0}})_4.$ Therefore, even if we add the restriction condition $\int_{T_\Gamma}\gamma\circ\Phi_{b_0}\,dv_g=0$, when taking the upper bound of the area functional, we cannot get a sharp estimate.
\end{remark}

\section*{\raggedright \textbf{Acknowledgements}}
We would like to thank Robert Bryant for his private correspondence, which provided valuable help in proving Lemma 2.1. We would like to thank my supervisor, Zhenlei Zhang, who introduced this problem to me and offered useful discussions. We are also grateful to the referees for their valuable comments and suggestions, which helped us improve the quality of the manuscript.

\bibliographystyle{amsalpha} 
\bibliography{main}

\end{document}